\documentclass[11pt]{article}
\usepackage{amssymb,amsmath,amsfonts,amsthm}
\usepackage{graphicx}
\usepackage{latexsym}
 \usepackage{epsfig} %%%FIG

\newtheorem{thm}{Theorem}[section]
\newtheorem{lem}[thm]{Lemma}

\newtheorem{cor}[thm]{Corollary}
\newtheorem{conj}[thm]{Conjecture}
\newtheorem{cl}[thm]{Claim}

\newtheorem{prob}[thm]{Problem}

\def\sat{{\text{\rm sat}}}
\def\ssat{{\text{\rm ssat}}}
\def\binom#1#2{{#1\choose#2}}

\def\cH{{\mathcal H}}

\def\qed
 {\ifhmode\unskip\nobreak\hfill$\Box$\medskip\fi
 \ifmmode\eqno{\Box}\fi}

\title{{\bf \huge Cycle-saturated graphs \\ with minimum number of edges}
    \footnote{ This copy was printed on {\today}.\quad  
    {\rm\small {\jobname}.tex,} \hfill Version as of 
March 01, 2011.
    \break\indent{\it Keywords:} graphs, cycles, extremal graphs, minimal saturated graphs. 
    \hfill \break\indent{\it 2010 Mathematics Subject Classification:} 05C38, 05C35.}}
\author{{\bf Zolt\'an F\"uredi}
\thanks{ Research supported in part by the Hungarian National Science Foundation
 OTKA, and by the National Science Foundation under grant NFS DMS 09-01276.} 
\\ Department of Mathematics, University of Illinois at Urbana-Champaign,
\\ Urbana, IL 61801, USA \quad and
\\ R\'enyi Institute of Mathematics of the Hungarian Academy of Sciences,
\\ Budapest, P. O. Box 127, Hungary-1364
\\ e-mail: \texttt{z-furedi@illinois.edu}\quad  and\quad \texttt{furedi@renyi.hu}
\\ and
\\ {\bf  Younjin Kim}
\\ Department of Mathematics, University of Illinois at Urbana-Champaign,
\\ Urbana, IL 61801, USA. 
\\ e-mail: \texttt{ykim36@illinois.edu}}
\date{${}$}

\begin{document}
\maketitle

\begin{abstract}
A graph $G$ is called $H$-saturated if it does not contain any
	copy of $H$, but for any edge $e$ in the complement of $G$
	the graph $G+e$  contains some $H$.
The minimum size of an $n$-vertex  $H$-saturated graph is
	denoted by  $\sat(n,H)$. 
We prove $$ 
	\sat(n,C_k) = n + n/k    + O((n/k^2) + k^2)$$ 
holds for all $n\geq k\geq 3$, where $C_k$ is a cycle with length $k$.
We have a similar result for semi-saturated graphs
	$$ \ssat(n,C_k) = n + n/(2k)    + O((n/k^2) + k).$$ 
We conjecture that our three constructions are optimal.
\end{abstract}

\section{A short history}

\noindent
    A graph $G$ is said to be   {${H}$-saturated} if\\
    --- it does not contain $H$ as a subgraph, but\\
   --- the addition of any new edge (from $E(\overline{G})$) creates a copy of $H$. \\
Let  $ {\rm sat} (n,H)$ denote the  
  {\it minimum} size of an $H$-saturated graph on $n$ vertices. 
Given $H$, it is difficult to determine $ {\rm sat} (n,H)$ because
 this function is not necessarily monotone in $n$, or in $H$.
Recent surveys are by J. Faudree, Gould, and Schmitt~\cite{fur:Faudree_survey}, 
 and by Pikhurko~\cite{Pik4}  on the hypergraph case.
It is known~\cite{fur:KT} that for every graph $H$ there exists a constant 
 $c_H$ such that 
         $$   {\rm sat} (n,H) < c_Hn$$
 holds for all $n$.
However, it is not known if the $ \lim_{n\to \infty}  {\rm sat} (n,H)/n$ exists;
 Pikhurko~\cite{Pik4} has an example of a four graph set $\cH$ when
  $\sat(n,\cH)/n$ oscillates, it does not tend to a limit. 

Since the classical theorem of Erd\H os, Hajnal, and Moon~\cite{fur:EHM}  
 (they determined ${\rm sat} (n,K_p)$ for all $n$ and $p$), 
 and  its generalization for hypergraphs by Bollob\'as~\cite{fur:boll65},
 there have been many interesting hypergraph results (e.g.,  Kalai~\cite{fur:Kal}, 
 Frankl~\cite{fur:frankl}, Alon~\cite{fur:Alon}, using Lov\'asz' algebraic method) 
but here we only discuss the graph case.

Remarkable asymptotics were given by Alon, Erd\H os, Holzman, and 
 Krivelevich~\cite{fur:AEHK,fur:EH} (saturation and degrees). 
Bohman, Fonoberova, and Pikhurko~\cite{fur:BFP}
 determined  the sat-function asymptotically 
 for a class of complete multipartite graphs.
More recently,
 for multiple copies of $K_p$
    Faudree, Ferrara, Gould, and Jacobson~\cite{fur:FFGJ} determined 
    $ {\rm sat} ( {tK_p}, n)$ for  $n \geq n_0(p,t)$.

\section{Cycle-saturated graphs}

What is the saturation number for the cycle, $C_k$?
This has been considered by various authors, 
  however, in most cases it has remained unsolved. 
Here relatively tight bounds are given.

\begin{thm} \label{th:1}
        \quad For all $k \geq 7$ and $n\geq 2k-5$
\begin{equation}\label{eq:th1}
  \left(1 + \frac{1}{k+2}\right)n - 1 < {\rm sat} (n,C_k) < 
    \left(1 + \frac{1}{k-4}\right)n + \binom{k-4}{2}. 
  \end{equation}
 \end{thm}

The construction giving the upper bound is presented at the end of this section,
 the proof of the lower bound (which works for all $n,k\geq 5$) is postponed to 
 Section~\ref{s:sat_lower}.

The case of $ {\rm sat}(n,C_3)=n-1$ is trivial; 
 the cases $k=4$ and $k=5$ were established by Ollmann~\cite{fur:Oll} in 1972 and  
 by   Ya-Chen~\cite{fur:yachen_C5} in  2009, resp. 
\begin{eqnarray}  
  {\rm sat} (n, {C_4})&=&\big\lfloor \frac{3n-5}{2}\big\rfloor \text{~~~~ for~~} n \geq 5.
   \notag % \label{fur:eqC4}
   \\
 {\rm sat} (n, {C_5})&=& \big\lceil \frac{10(n-1)}{7} \big\rceil \text{~~~~ for~~} n \geq 21.
  \label{fur:eqC5}
  \end{eqnarray}
Actually, (\ref{fur:eqC5}) was conjectured by Fisher, Fraughnaugh, Langley~\cite{fur:FFL}. 
Later Ya-Chen~\cite{fur:yachen_C5all} determined  ${\rm sat} (n, {C_5})$ for all
 $n$, as well as all extremal graphs.

The best previously known general lower bound came from 
  Barefoot, Clark, Entringer, Porter, Sz\'ekely, and Tuza~\cite{BCE}, 
   and the best upper bound (a clever, complicated construction resembling a bicycle wheel) 
  came from Gould, {\L}uczak, and Schmitt~\cite{fur:GLS} 
\begin{equation}\label{eq:old}
   \left( 1+ \frac{1}{2k+8} \right) n \leq   
    {\rm sat} (n,C_k)  \leq \left( 1 + \frac{2}{k- \varepsilon(k) }\right)n + O(k^2)
   \end{equation}
  where $\varepsilon(k) = 2$ for $k$ even $\geq 10$, $\varepsilon(k) = 3$ for $k$ odd $\geq 17$.
Although there is still a gap, Theorem~\ref{th:1} supersedes all earlier results
 for $k\geq 6$ except the construction giving  
  ${\rm sat} (n, {C_6}) \leq \frac{3}{2}n$  for $n \geq 11$ from~\cite{fur:GLS}.

Our new construction for a $k$-cycle saturated graph for $n = (k-1)+ t(k-4)$
  can be read from the picture below.
    \begin{figure}[htp] 
        \centering
    \includegraphics[scale=0.44]{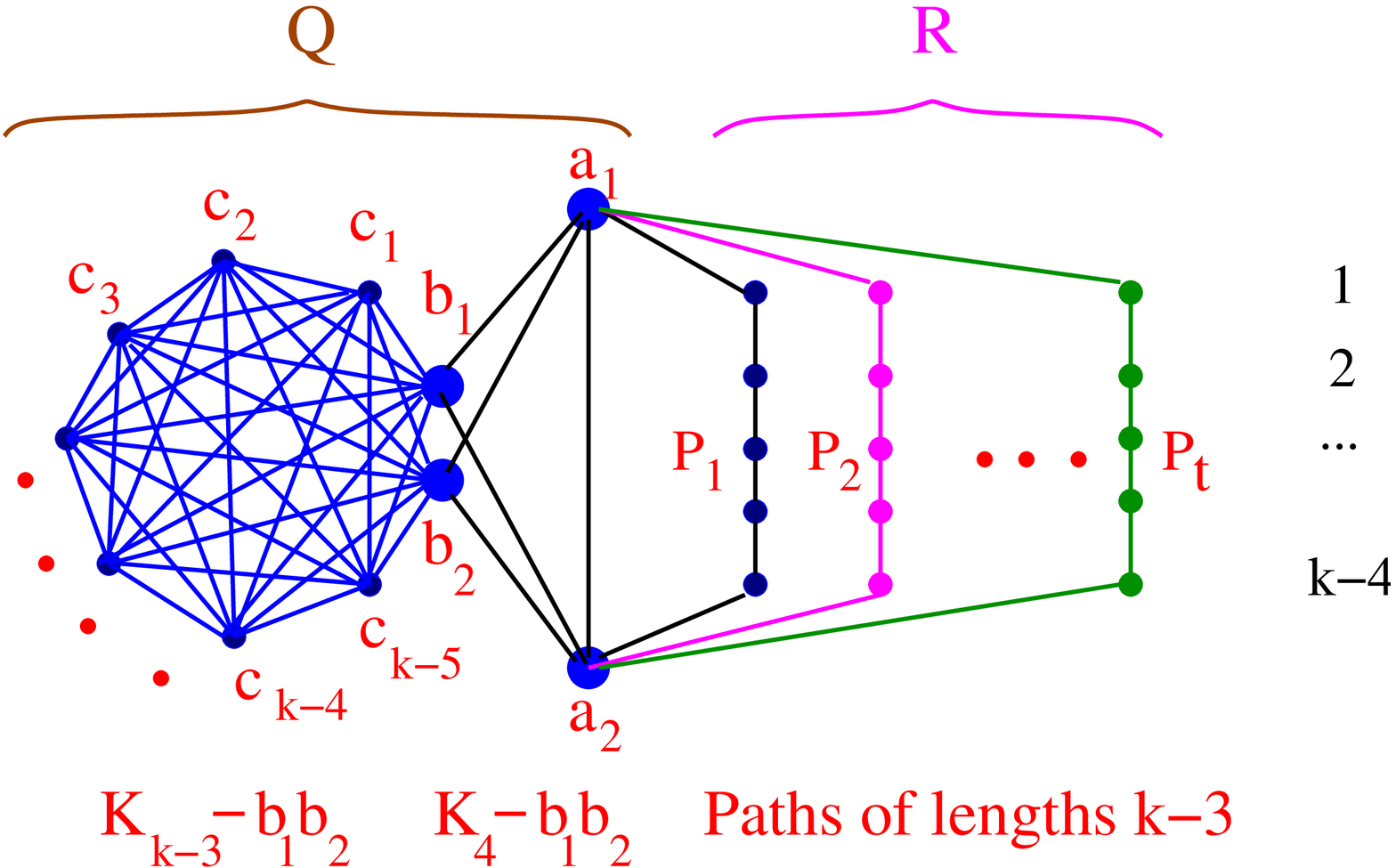} 
       % \vskip -1.3cm
    \end{figure}
\noindent

To be precise, define the graph $H:=H_{k,n}$ on $n$ vertices,
  for arbitrary $n> k\geq 7$ as follows.
Write $n$ in the form 
 $$n=(k-1)+r+t(k-4)
 $$ 
 where $t\geq 1$ is an integer and  $0\leq r \leq k-5$. 
The vertex set $V(H)$ consists of the pairwise disjoint sets
 $A$, $B$, $C$, $D$, and $R_i$ for $1\leq i\leq t$, 
$V(H) = A \cup B \cup C \cup D\cup R_1 \cup R_2 \cup \cdots \cup R_t$ 
 where $|A|=|B|=2$, $|C|=k-5$, $|D|=r$,
  and 
  $|R_1|=|R_2|= \cdots = |R_t|= k-4$ and
   $A = \{ a_1, a_2 \}$, $B= \{ b_1, b_2 \}$,  $C= \{ c_1,c_2, \cdots , c_{k-5} \}$,
 $D= \{ d_1,d_2, \cdots , d_{r} \}$,  
 $R_\alpha= \{ r_{\alpha,1},r_{\alpha,2}, \dots , r_{\alpha, k-4} \}$.  
We also denote $A\cup B\cup C\cup D$ by $Q$ and $R_1\cup \dots\cup R_t$ by $R$.
 
The edge set of $H$ does not contain $b_1b_2$ and it 
 consists of an almost complete graph $K_{k-3}$ minus an edge on $C\cup B$, 
 a $K_4$ minus an edge on $B\cup A$, 
 $r$ pending edges connecting $c_i$ and $d_i$,  
 and $t$ paths $P_\alpha$ of length $k-3$ with vertex sets  $A \cup R_\alpha$
 with endpoints $a_1$ and $a_2$. 
The number of edges
$$
  |E(G)|=    \binom{k-3}{2}+4+r+t(k-3). 
  $$ 
It is not difficult to check that, indeed, $H$ is $C_k$-saturated (See details in Section~\ref{s:Hk}).
After which, a little calculation yields the upper bound in (\ref{eq:th1}). 

We strongly believe that this construction is essentially optimal.
 \begin{conj}\label{conj:1}
 There exists a $k_0$ such that
 ${\rm sat}(n,C_k)= \left(1+ \dfrac{1}{k-4}\right)n+O(k^2) $
holds for each $k> k_0$. 
 \end{conj}

\section{The graph $H_{k,n}$ is $C_k$-saturated,\\ the proof of the upper bound for $\sat(n,C_k)$}
\label{s:Hk}

First we check that $H:=H_{k,n}$ is $C_k$-free.
If a cycle with vertex set $Y$ is entirely in $Q$, then it is contained in $A\cup B\cup C$,
 so $|Y|\leq k-1$. 
If $Y$ contains a vertex $r_{\alpha,i}$ then $A\cup R_\alpha\subset Y$, the $k-3$ edges 
 of the path $P_\alpha$ are part of the cycle.
However, it is impossible to join $a_1$ and $a_2$ by a path of length 3, so $|Y|\neq k$.

The key observation to know that $H$ is $C_k$-saturated is that $a_1$ and $a_2$  are connected 
 inside $Q$ by a path $T_\ell$ of any other lengths $\ell$ except for $3$
\begin{equation}\label{eq:Tl}
  \exists\text{ path } T_\ell\subset Q: \ell\in \{ 1,2,4,5,\dots, k-3,k-2\}\, 
  \text{\rm  with endpoints }a_1, \, a_2.
  \end{equation}
For example, $T_1$ is $a_1a_2$, $T_2=a_1b_1a_2$, $T_4=a_1b_1c_1b_2a_2$, etc. 
Also the vertices 
  $a_i$ ($i=1,2$) and $q\in Q\setminus \{ a_i\}$ are connected by a path $U^i(m)$
 of length $m$ inside $Q$ for $\lceil (k+1)/2\rceil \leq m\leq k-2$.
 \begin{equation}\label{eq:Um}
  \exists\text{ path } U^i(m)\subset Q: m\in \{\lceil (k+1)/2\rceil ,\dots, k-3,k-2\}\, 
   \text{\rm  with endpoints }a_i, \, q\in Q.
  \end{equation}
Note that this is true for any $m\geq 4$ but we will apply (\ref{eq:Um}) only for
 $\lceil (k+1)/2\rceil\geq 4$.

Now add an edge $e$ to $H$ from its complement. 
We distinguish  four disjoint cases.
\newline\noindent${}$Case 1.\quad If $e$ is contained in the induced cycle $A\cup R_\alpha$
 then we get a path connecting $a_1$ and $a_2$ in $A\cup R_\alpha$ of length 
 $t$, where $t$ is at least two and at most $k-4$.
This path with $T_{k-t}$ form a $k$-cycle.
\newline\noindent${}$Case 2.\quad If the endpoints of $e$ are $r_{\alpha,i}$ and $r_{\beta,j}$ 
 with $\alpha\neq \beta$ then we may suppose that $1\leq i\leq j\leq k-4$.
The vertex $r_{\alpha,i}$ splits the path $P_\alpha$ into two parts, $P_\alpha^1$ and $P_\alpha^2$,
 where $P_\alpha^1$ starts at $a_1$ and has length 
 $i$, and $P_\alpha^2$ ends at $a_2 $ and has length $k-3-i$.
Consider the path $\pi:=P_\alpha^1eP_\beta^{2}$, its length is $k-2-j+i$.
This length is between 3 and $k-2$ so we can apply (\ref{eq:Tl}) to add an appropriate 
 $T_{j-i+2}$ to complete $\pi$ to a $k$-cycle unless $j-i+2=3$.
 In the latter, the edge $a_1a_2$ together with $P_\beta^1$, $e$, and $P_\alpha^2$ form a $C_k$.
\newline\noindent${}$Case 3.\quad If the endpoints of $e$ are $r_{\alpha,i}$ 
 and  $q\in B\cup C\cup D$, then again by symmetry, we may suppose that
 $i\leq (k-3)/2$, so the length of $P_\alpha^1$ is at most $\lfloor (k-3)/2 \rfloor$.
Then, by (\ref{eq:Um}) there is an $U^1(m)$ so that $P_\alpha^1$, $e$ and $U^1(m)$ form a $k$-cycle.
\newline\noindent${}$Case 4.\quad Finally, $e$ is contained in $Q$.\\
For $e=a_1c_1$  we use $P_1$ to get the $k$-cycle $a_1c_1b_1a_2P_1$, \\ 
for $e=a_1d_1$ we have the $k$-cycle $d_1c_1c_2\dots c_{k-5}b_2a_2b_1a_1$, \\
for $e=b_1b_2$ we have to use $P_1$, i.e., here we need again that $t\geq 1$, \\
for $e=b_1d_1$ we have the $k$-cycle $d_1c_1c_2\dots c_{k-5}b_2a_2a_1b_1$, \\
for $e=c_1d_2$ we have the $k$-cycle $c_1d_2c_2\dots c_{k-5}b_2a_2a_1b_1$, finally\\
for $e=d_1d_2$ we have the $k$-cycle $c_1d_1d_2c_2\dots c_{k-5}b_2a_2b_1$. \qed

\section{Semisaturated graphs}

 A graph $G$ is   {$H$-\bf semisaturated} (formerly called {\em strongly} saturated) 
  if  $G+e$ contains  more copies of $H$ than $G$ does for 
   $\forall e \in E(\overline{G})$. 
Let $ {\rm ssat}(n,H)$ be the  minimum size of an $H$-semisaturated graph.
Obviously, $ {\rm ssat} (n,H) \leq  {\rm sat} (n,H).$

It is known that  $  { {\rm ssat} (n, K_p)} =  {\rm sat} (n, K_p) $
    (it follows from Frankl/Alon/Kalai  generalizations of Bollob\' as set pair theorem) 
   and   $ { {\rm ssat} (n, C_4)} =  {\rm sat} ( n, C_4)$ (Tuza~\cite{Tuza89}).
Below we have a $C_5$-semisaturated graph on $1+8t$ vertices and $11t$ edges. 
Every vertex can be reached by a path of length 2 from $y$. 
         \begin{figure}[htp]
    \includegraphics[scale=0.4]{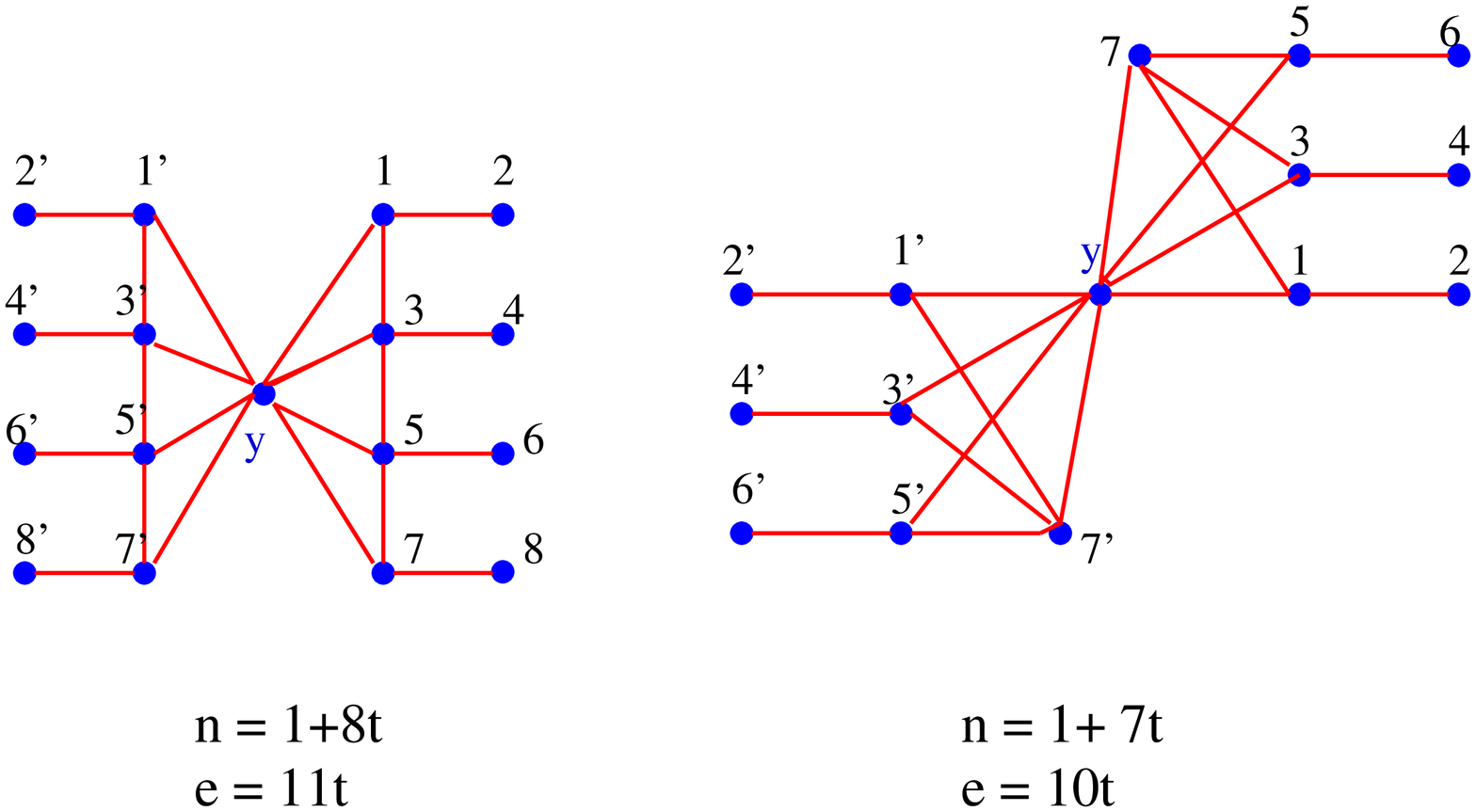}
        \end{figure}
Joining one, two or three triangles to the central vertex $y$
 one obtains $C_5$ semisaturated graphs with $8t+3$, $8t+5$, or $8t+7$
 vertices and $11t+3$, $11t+6$, or $11t+9$ edges, resp. 
Leaving out a pendant edge, we can extend these constructions for
 even values of $n$ 
\begin{equation}\label{eq:C5}
 \ssat(n,C_5)\leq \big\lceil \frac{11}{8}(n-1)\big\rceil
    \, \text{ for all \, } n\geq 5.
 \end{equation}
The picture on the right is the extremal $C_5$-saturated graph by (\ref{fur:eqC5}).
\begin{conj}
       \quad $  {\rm ssat} (n,C_5) = \dfrac{11}{8}n + O(1)$. 
Maybe equality holds in  {\rm (\ref{eq:C5})} for $n> n_0$.  
 \end{conj}

Since   $11/8=1.375  < 10/7 = 1.42...$ 
  inequalities (\ref{fur:eqC5}) and (\ref{eq:C5}) imply that  
$${\rm ssat} (n,C_5) < 
    {\rm sat} (n,C_5) \, \text{ for all \, } n\geq 21.
 $$   
Our next Theorem shows that a similar statement holds for every cycle $C_k$
 with $k> 12$ (and probably for $k\in \{ 6, 7, \dots, 12\}$, too). 

\begin{thm} \label{th:2}
        \quad For all $n\geq k \geq 6$ 
\begin{equation}\label{eq:th2}
  \left(1 + \frac{1}{2k-2}\right)n - 2 < {\rm ssat} (n,C_k) < 
    \left(1 + \frac{1}{2k-10}\right)n + k-1. 
  \end{equation}
 \end{thm}

The proof of the lower bound is postponed to Section~\ref{s:ssat_lower3}. 
The construction yielding the upper bound is presented in the next two sections
 where we describe a way to improve the $O(k)$ term as well as
 give better constructions for $k=6$. 
We believe that our constructions are essentially optimal.
 \begin{conj}\label{conj:2}\quad 
 There exists a $k_0$ such that \enskip 
 ${\rm ssat}(n,C_k)= \left(1+ \dfrac{1}{2k-10}\right)n+O(k) $\enskip
holds for each $k> k_0$. 
 \end{conj}

\section{Constructions of sparse $C_k$-semisaturated graphs}

In this section we define an infinite class of $C_k$-semisaturated graphs, 
 $H_{k,n}^2$ (more precisely $H_{k,n}^2(G)$).
  
Call a graph $G$ $k$-{\it suitable} with special vertices $a_1$ and $a_2$ if 
\newline\noindent${}$(S1)\quad $G$ is $C_k$-semisaturated, 
\newline\noindent${}$(S2)\quad $\exists$ a path  $T_\ell$ in $G$ with endpoints  
 $a_1$ and $a_2$ and of length $\ell$ for all $1\leq \ell\leq k-2$, and
\newline\noindent${}$(S3)\quad 
 for every  $q\in V(G)\setminus \{ a_1, a_2\}$, and integers
 $m_1$ and $m_2$ with  $m_1+m_2=k$ and $2 \leq m_i \leq k-2$ 
  \newline\noindent${}$\quad\quad\enskip 
  $\exists$ an $i\in \{ 1, 2 \}$ and a path $U(a_i,q,m_i)$ of length 
  $m_i$ and with endpoints $a_i$ and $q$. 

For example, it is easy to see, that a {\bf wheel} with $r$ spikes  $W_k^r$ is such a graph, 
 $k\geq r$, $k\geq 4$. 
It is defined by the $(k+r)$-element vertex set 
  $\{ a_1, a_2, \dots, a_{k}, d_1, \dots, d_r\}$ and by $2k-2+r$ edges joining
  $a_1$ to all other $a_i$'s, forming a cycle $a_2a_3\dots a_{k}$ of length $k-1$, 
  and joining each $d_i$ to $a_i$. 

Define the graph $H_{k,n}^2(G)$ as follows, when 
 $n$ is in the form 
 $$n=|V(G)|+t(k-3)
 $$ 
 where $t\geq 0$ is an integer. 
The vertex set $V(H)$ consists of the pairwise disjoint sets
 $Q$ and $R_i$ for $1\leq i\leq t$, 
$V(H) = Q \cup R_1 \cup \dots \cup R_{t}$ 
 where $|Q|=|V(G)|$, $|R_1|=|R_2|= \cdots = |R_t|= k-3$ and
   $A:=\{ a_1, a_2 \}\subset Q$. 
The edge set of $H$ consists of a copy of $G$ with vertex set $Q$, 
  and $t$ paths with endpoints $a_1$ and $a_2$ and vertex sets $A\cup R_\alpha$. 
The number of edges is 
\begin{equation*}  % \label{eq:C2}
  |E(H)|=  |E(G)|+t(k-2).  
  \end{equation*} 
It is not difficult to check that, indeed, $H$ is $C_k$-semisaturated, 
 the details are similar (but much simpler) to those in  Section~\ref{s:Hk}, 
 so we do not repeat that proof.

Finally, considering $H_{k,n}^2(W_k^r)$ (where now $4\leq r\leq k$) we obtain that
 for all $n\geq k+4$
\begin{equation}\label{eq:6}
    {\rm ssat} (n,C_k) \leq  n+
    \big\lfloor \frac{n-7}{k-3}\big\rfloor + k-3. 
  \end{equation}
\begin{cor}\label{cor:C5}\quad 
 ${\rm ssat} (n,C_6) \leq 
    \big\lceil \dfrac{4}{3}n\big\rceil.$ 
  \end{cor}

\section{Thinner constructions of sparse $C_k$-semisaturated graphs}

In this section we define another infinite class of $C_k$-semisaturated graphs, 
 $H_{k,n}^3$ (more precisely $H_{k,n}^3(G)$)
 yielding the upper bound (\ref{eq:th2}) in Theorem~\ref{th:2}.

Call a graph $G$ $\{ k. k+2\}$-{\it suitable} with special vertices $a_1$ and $a_2$ if 
 (S1) and (S2) hold but (S3) is replaced by the following
\newline\noindent${}$(S3)${}^+$\quad 
 for every  $q\in V(G)\setminus \{ a_1, a_2\}$, and integers $m_1, m_2$
  either there exists a path $U(a_1,q,m_1)$ 
 \newline\noindent${}$\quad\quad\enskip   
  (of length $m_1$ and with endpoints $a_1$ and $q$) 
  or a path $U(a_2,q,m_2)$ in the following cases
 \newline\noindent${}$\quad\quad\enskip    $m_1+m_2=k$ and $3 \leq m_i \leq k-3$, 
 \newline\noindent${}$\quad\quad\enskip    $m_1+m_2=k+2$ and $4 \leq m_i \leq k-4$.  
 \newline
It is easy to see, that the wheel $W_k^r$ with $r$ spikes is such a graph, 
 $k\geq r\geq 0$, $k\geq 4$.

Define the graph $H_{k,n}^3(G)$ as follows, when 
 $n$ is in the form 
\begin{equation}\label{eq:n}
  n=|V(G)|+t(2k-10)-r
 \end{equation} 
 where $t\geq 2$ is an integer and $0\leq r < 2k-10$. 
The vertex set $V(H)$ consists of the pairwise disjoint sets
 $Q$, $R_i$ and $D$ for $1\leq i\leq t$, 
$V(H) = Q \cup R_1 \cup \dots \cup R_{t}\cup D$ 
 where $|Q|=|V(G)|$, $|R_1|=|R_2|= \cdots = |R_t|= k-5$, $|D|=t(k-5)-r$ and
   $A:=\{ a_1, a_2 \}\subset Q$. 
The edge set of $H$ consists of a copy of $G$ with vertex set $Q$, 
  and $t$ paths with endpoints $a_1$ and $a_2$ and vertex sets $A\cup R_\alpha$
   and finally $|D|$ spikes, a matching with edges from $\cup R_\alpha$ to $D$.

  \begin{figure}[htp]
        \centering
    \includegraphics[scale=0.4]{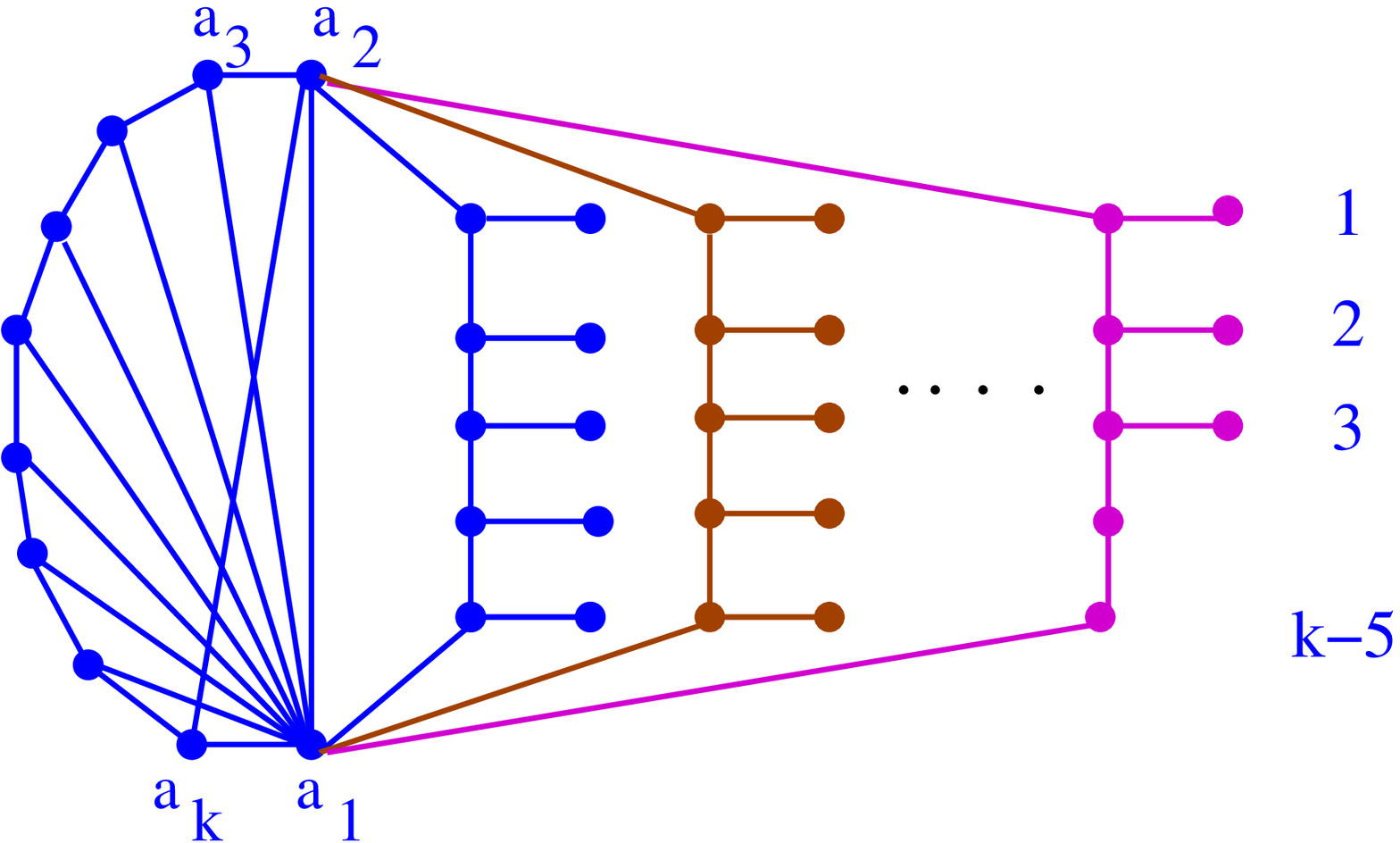}
    \end{figure}

The number of edges is
\begin{equation}\label{eq:C3}
  |E(H)|=  |E(G)|+t(2k-9)-r.  
  \end{equation} 
It is not difficult to check that $H$ is $C_k$-semisaturated, 
 the details are  similar (but simpler) to those in  Section~\ref{s:Hk}.
As an example we present one case.

Add the edge $qd$ to $H$ where $q\in V(G)\setminus \{ a_1, a_2\}$ and $d\in D$.
Let us denote the (unique) neighbor of $d$ by $x$, $x\in R_\alpha$.  
The distance of $x$ to $a_1$ is denoted by $\ell$. 
Then the length of the $qdx\dots a_1$ path is $\ell+2\geq 3$ and the
    length of the $qdx\dots a_2$ path is $(k-4-\ell)+2\geq 3$ and one can find a 
    $C_k$ through $qd$ using property (S3)${}^+$.

Considering $H_{k,n}^3(W_k)$ (with $t\geq 2$) we obtain from (\ref{eq:n}) and (\ref{eq:C3}) that 
 for all $n\geq 3k-9$
\begin{equation}\label{eq:10}
    {\rm ssat} (n,C_k) \leq 
    \big\lceil \left(1 + \frac{1}{2k-10}\right)(n-k)\big\rceil + 2k-2. 
  \end{equation}
Using $H^2(k,n)$, it is easy to see that (\ref{eq:10}) holds for all $n\geq k$, leading to
 the upper bound in (\ref{eq:th2}).

One can slightly improve (\ref{eq:6}) and (\ref{eq:10}) if there are 
 special graphs thinner than the wheel $W_k$. 
\begin{prob}\label{prob:special}\quad 
Determine $s(k)$, the minimum size of a $k$-vertex $k$-special graph (i.e., 
one satisfying {\rm (S1)--(S3)}).
Determine $s'(k)$, the minimum size of a $k$-vertex $\{k, k+2\}$-special graph (i.e., 
one satisfying {\rm (S1), (S2)} and {\rm (S3)${}^+$}).
  \end{prob}

\section{Degree one vertices in (semi)saturated graphs}\label{s:ssat_lower1}

Suppose that $G$ is a $C_k$-semisaturated graph where $k\geq 5$, $|V(G)|= n \geq k$. 
Obviously, $G$ is connected. 
Let $X$ be the set of vertices of degree one,  $X := \{ v \in V(G) : \deg_G(v)=1 \}$, 
 its size is $s$ and its elements are denoted as $X = \{ x_1, x_2, \cdots , x_s \}$. 
Denote the neighbor of $x_i$ by $y_i$, $Y:=\{y_1, \dots, y_s\}$ and let $Z:=
 V(G)\setminus (X\cup Y)$.
We also denote the neighborhood of any vertex $v$ by $N_G(v)$ or briefly by $N(v)$. 

\begin{lem}\label{le:degree1} {\rm  (The neighbors of degree one vertices.)}\newline
(i) \quad    $y_i \neq y_j$ for $1 \leq i \neq j \leq s$, so $|Y|=|X|$.   \newline
(ii) \enskip $\deg(y)\geq 3$ for every $y\in Y$,  \newline
(iii) \,       if  $\deg_G(x)=1$, then  $G - \{ x \}$ is also a $C_k$-semisaturated  graph. 
\end{lem}

\noindent
{\sl Proof.}\quad
If $y_i=y_j$,  then the addition of  $x_ix_j$ to $G$ does not create a new $k$-cycle.
If $\deg(y_i)=2$ and $N(y_i)=\{ x_i,w\}$, the addition of $x_iw$ to $G$ does not 
 create a new $k$-cycle.
Finally, (iii) is obvious. \qed

Split $Y$ and $Z$ according to the degrees of their vertices. Thus, divide $V(G)$ into 
 five parts $\{X, Y_3, Y_{4+}, Z_2, Z_{3+} \}$, \newline\indent
$ Y_3 : = \{ v \in Y : \deg_G(v)=3\}$ and \enskip
$ Y_{4+} : = \{ v \in Y : \deg_G(v) \geq 4 \}$, \newline\indent
$ Z_2 : = \{ v \in Z : \deg_G(v)=2  \}$ and \enskip
$ Z_{3+} : = \{ v \in Z : \deg_G(v) \geq 3\}$.

\begin{lem}\label{le:structure} {\rm  (The structure of $C_k$-saturated graphs. See~\cite{BCE})}. \newline
 Suppose that $G$ is a $C_k$-saturated graph (and $k\geq 5$). Then\newline 
(iv) \quad  if  $x_iy_iw$ is a path in $G$ (with $x_i\in X$, $y_i\in Y$), then
  $\deg(w)\geq 3$.
  So there are no edges from $Z_2$ to $Y$ (or to $X$).
     \newline
(v) \enskip  If $y_iy_j$ is an edge of $G$ (with $y_i, y_j\in Y$), then
$\deg(y_i)\geq 4$. 
  So there are no edges in $Y_3$, no edges from $Y_3$ to $Y_4$.
In other words, every $y\in Y_3$ has one neighbor in $X$ and two in $Z_{3+}$. 
     \newline
(vi) \enskip 
The induced graph $G[Z_2]$ consists of paths of length at most $k-2$. 
\qed
\end{lem}

  \begin{figure}[htp]
        \centering
    \includegraphics[scale=0.6]{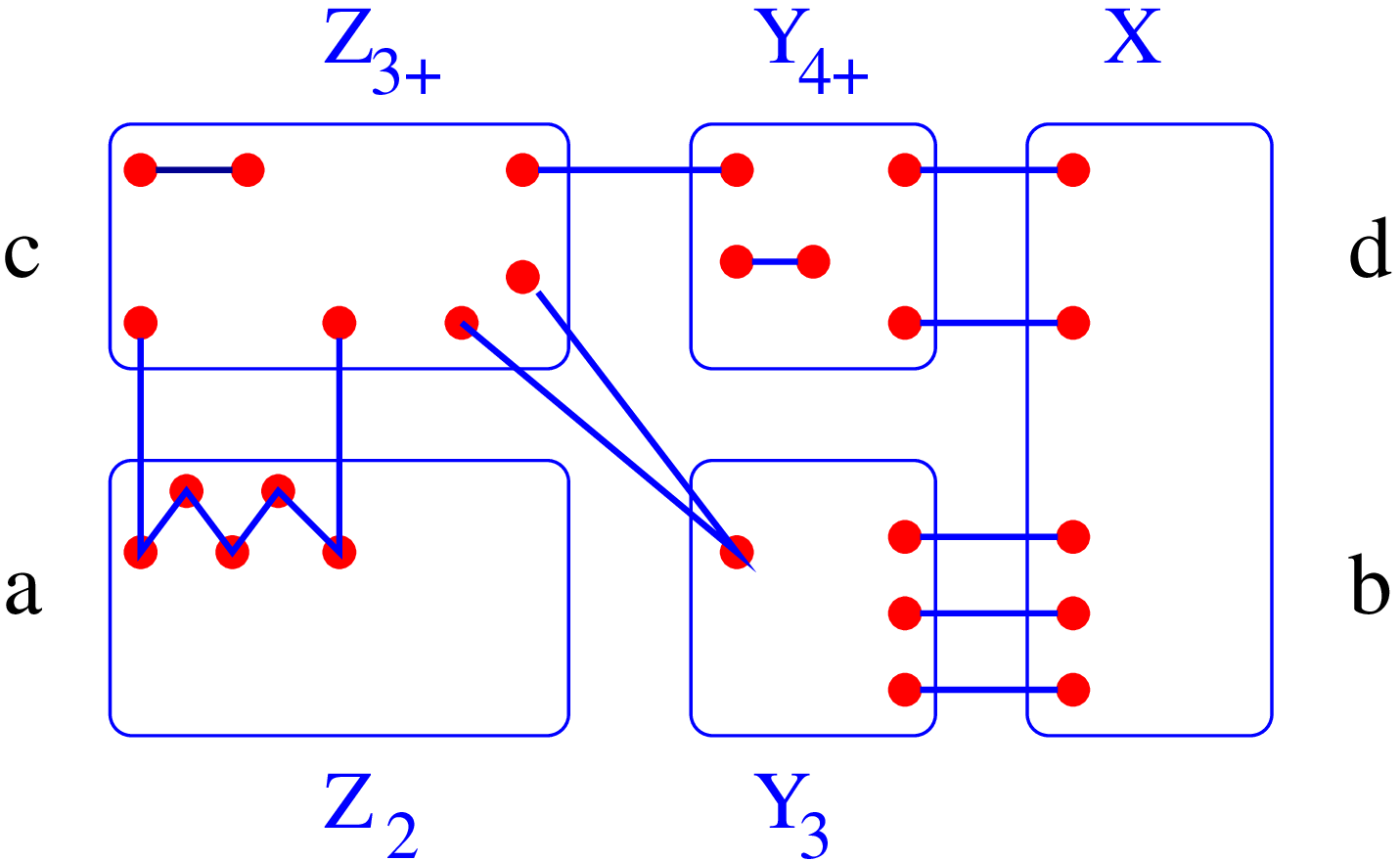}
    \end{figure}

\section{Semisaturated graphs without pendant edges}\label{s:sat_lower2}

\begin{cl}\label{cl:8}
Suppose that $G$ is a $C_k$-semisaturated graph on $n$ vertices
 with minimum degree at least $2$, $k\geq 5$. 
Then every vertex $w$ is contained in some cycle of length at most $k+1$.
\end{cl}
\begin{proof}
Consider two arbitrary vertices $z_1,z_2$ in the neighborhood $N(w)$. 
If $z_1z_2 \in E(G)$, then $w$ is contained in a triangle. 
If $z_1z_2\not\in E(G)$, then $G+z_1z_2$ contains a new $k$-cycle;
 there is a path $P$ of length $(k-1)$ in $G$ with endpoints $z_1$ and $z_2$. 
If $P$ avoids $w$, then $P$ together with $z_1wz_2$ form a $k+1$ cycle.
If $w$ splits $P$ into two paths
 $L_1, L_2$, where $L_i$ starts in $z_i$ and ends in $w$, then either  
  $L_1 +z_1w$, or $L_2+z_2w$, or both form a proper cycle of 
 length at most $k-1$.
\end{proof}

Note that Claim~\ref{cl:8} itself (and the connectedness of $G$)
 immediately imply 
$$
 e(G)\geq (n-1)\dfrac{k+2}{k+1}.
  $$ 
We can do a bit better repeatedly using  the semisaturatedness of $G$.

\begin{lem}\label{le:lower}
Suppose that $G$ is a $C_k$-semisaturated graph on $n$ vertices
 with minimum degree at least $2$, $k\geq 5$. 
Then
\begin{equation*} %\label{eq:81}
   e(G)\geq \frac{k}{k-1}\, n -\frac{k+1}{k-1}. 
 \end{equation*}
\end{lem}

\noindent
{\sl Proof.}\quad
We define an increasing sequence of subgraphs $G_1, G_2, \dots, G_t=G$ 
 with vertex sets $V_1\subseteq V_2 \subseteq \dots \subseteq V_t=V(G)$  
 such that $G_i$ is a subgraph of $G_{i+1}$ and 
\begin{equation}\label{eq:13}
   |E(G_{i+1})\setminus E(G_i)|\geq \frac{k}{k-1} \left(|V_{i+1}|-|V_i|\right)
   \end{equation}
(for $i=1,2, \dots, t-1$).
This, together with 
\begin{equation}\label{eq:14}
 e(G_1)\geq  \frac{k}{k-1}\, |V_1| -\frac{k+1}{k-1}
  \end{equation}
imply the claim.

$G_1$ is the shortest cycle in the graph $G$. 
Its length is at most $k+1$ so (\ref{eq:14}) obviously holds.

If $G_i$ is defined and one can find a path $P$ of length at most $k$
 with endpoints in $V_i$ but $E(P)\setminus E(G_i)\neq\emptyset$, then
 we can take $E(G_{i+1})=E(G_i)\cup E(P)$. 
From now on, we suppose that such a {\it short returning} 
 path does not exist.
Our procedure stops if $V(G_i)=V(G)=:V$. 

In the case of $V\setminus V_i\neq \emptyset$,
 the connectedness of $G$ 
 implies that there exists an $xy$ edge with $x\in V_i$ and
 $y \in V\setminus V_i$. 
Since $|N(y)|\geq 2$ we have another edge $yz\in E(G)$, $z\neq x$.

We have $N(y)\cap V_i=\{ x\}$, otherwise one gets a path $xyz$
 of length smaller than $k$ with endpoints in $V_i$ but going out of
 $G_i$, contradicting our earlier assumption. 
Similarly, we obtain that $N(y)$ contains no edge, otherwise we can 
 define $E(G_{i+1})$ as either $E(G_i)$ plus the three edges of a  
 triangle $xy$, $yz$, $xz$ or we add 
 four edges $xy$, $yz_1$, $yz_2$, and $z_1z_2$ but only three vertices
 (namely $y$, $z_1$, and $z_2$). The obtained 
 $G_{i+1}$ obviously satisfies (\ref{eq:13}) in both cases.
Similarly, if there is a cycle $C$ of length at most $k-1$ contaning 
 $y$, then we can define $E(G_{i+1})$ as $E(G_i)$ plus $E(C)$ and $xy$.
From now on, we suppose that such a {\it short cycle through} $y$ 
 does not exist.

Fix a neighbor $z$ of $y$, $z\neq x$.
Since $zx \not \in E(G)$, $G$ contains a path $P$ of length $k-1$ with 
  endvertices $x$ and $z$. 
Since $G$ does not contain a short returning path nor a short cycle through
 $y$, we obtain that $P$ avoids $y$ and $V(P)\cap V_i=\{ x\}$.

If the cycle $C:= P+xy+yz$ of length $k+1$ has any diagonal edge 
 then $G_{i+1}$ is obtained by adding $C$ together with its diagonals.
From now on, we suppose that $C$ does not have any diagonals.
More generally, if there is any {\it diagonal} {\it path} $P$ of length 
 $\ell \leq k-1$
 with edges disjoint from $E(G_i)\cup E(C)$ but with endpoints
 in $V_i\cup V(C)$ then we can define $E(G_{i+1}):=E(G_i)\cup E(C)\cup E(P)$
 and have added $k+\ell-1$ vertices and $k+\ell+1$ edges, 
 obviously satisfying  (\ref{eq:13}).

However, such a diagonal path exists. Let $w \neq y$ be the other neighbor 
 of $x$ in $C$. 
Since  $wz \not \in E(G)$,
 there is a path $P'$ of length $k-1$ with endpoints $w$ and $z$.
This $P'$ must have edges outside $E(G_i)\cup E(C)$ so 
 it can be shortened to a diagonal path $P$ of length at most $k-1$.
This completes the proof of the Lemma. \qed

\section{A lower bound for the number of edges of semisaturated graphs}\label{s:ssat_lower3}

In this section we finish the proof of Theorem~\ref{th:2}.
Let $G$ be a $C_k$-semisaturated graph on $n$ vertices
 with minimum number of edges, $k\geq 5$. 
Let $X$ be the set of degree one vertices, $x:=|X|$. 
By Lemma~\ref{le:degree1} $|X|\leq n/2$, and for $G':=G\setminus X$
   we have $e(G')=e(G)-x$ and $G'$ is a $C_k$-semisaturated graph on $n-x$ vertices
 with minimum degree at least 2.
Then Lemma~\ref{le:lower} can be applied to $e(G')$.
We obtain
\begin{eqnarray*}
\ssat(n,C_k)=e(G) &\geq& x + (n-x)\frac{k}{k-1} -\frac{k+1}{k-1} \\ % \label{eq:16} \\
           &\geq& \frac{n}{2} +\frac{n}{2} \frac{k}{k-1}-\frac{k+1}{k-1} % \notag
 =n\left(1+\frac{1}{2k-2}\right) -\frac{k+1}{k-1} 
\end{eqnarray*}
\qed

Since $\sat(n,C_k) \geq \ssat(n,C_k)$, this is already a better lower bound 
 than the one in~(\ref{eq:old}) from~\cite{BCE}.

\section{A lower bound for the number of edges of  $C_k$-saturated graphs}\label{s:sat_lower}

In this section we finish the proof of Theorem~\ref{th:1}.
Let $G$ be a $C_k$-saturated graph on $n$ vertices, $k\geq 5$. 
Let us consider the partition of $V(G)=X\cup Y_3\cup Y_{4+}\cup Z_2\cup Z_{3+}$ defined in 
 Section~\ref{s:ssat_lower1}, where $X$ is the set of degree one vertices,  
 $Y$ is their neighbors.
By Lemma~\ref{le:degree1} $|X|=|Y|$. 
To simplify notations we use $a:=|Z_2|$, $b:=|Y_3|$, $c:=|Z_{3+}|$, and $d:=|Y_{4+}|$.
We have
\begin{equation*}
   n= a+2b+c+2d.
  \end{equation*}
By definition of the parts we have the lower bound
$$ 2e(G) = \sum_{v\in V} \deg(v) \geq |X| + 2|Z_2| + 3|Y_3| + 3|Z_{3+}| + 4|Y_{4+}|.$$ 
This yields
\begin{equation}\label{eq:102}
 2e\geq 2n + c + d.
\end{equation}

Now we estimate the number of edges by considering four disjoint subsets of $E(G)$.
The part $X$ is adjacent to $|X|$ edges, 
and according to Lemma~\ref{le:structure}, $Z_2$ is adjacent to at least 
 $\frac{k}{k-1}|Z_2|$ edges, $Y_3$ is adjacent to exactly $3|Y_3|$ edges from which
   $|Y_3|$ has already been counted at $X$,
  and finally $Y_{4+}$ is adjacent to at least another $\frac{3}{2}|Y_{4+}|$ edges.
We obtain
$$ e(G) \geq |X|+ \frac{k}{k-1}|Z_2| + 2|Y_3| + \frac{3}{2}|Y_{4+}|. $$
Therefore we get 
\begin{equation}\label{eq:103}
	e \geq n + \frac{1}{k-1}a + b - c + \frac{1}{2}d.
\end{equation}

By Lemma~\ref{le:degree1} $G \setminus X$ is also $C_k$-semisaturated.
Apply Lemma~\ref{le:lower} to estimate $e(G\setminus X)=e-b-d$,  
 multiply by $(k-1)$ and rearrange, we get
\begin{equation}\label{eq:104}
	(k-1)e \geq kn - b -d -(k+1).
\end{equation}

Adding up the above three inequalities (\ref{eq:102}), (\ref{eq:103}), and (\ref{eq:104})
 we obtain 
\begin{equation*}  %\label{eq:10f}
	(k+2)e \geq (k+3)n  + \frac{1}{k-1}a + \frac{1}{2}d -(k+1).
\end{equation*}
This implies the desired lower bound in (\ref{eq:th1}). \qed

\noindent	
{\bf Remark.}\enskip 
We can do  slightly better if we multiply 
  (\ref{eq:102}), (\ref{eq:103}), and (\ref{eq:104})
by $k$, $k-1$, and $k-3$, resp., then adding up and simplifying 
 we get 
 \begin{equation}\label{eq:10ff}
	e(G)  
%%% \geq \left(1+ \frac{k-2}{(k-2)(k+1)+4}\right)n  - \frac{k^2-2k-3}{k^2-k+4}.
       > \frac{k^2}{k^2-k+2}\, n -1. 
\end{equation}

\end{document}